\documentclass[12 pt]{article}

\usepackage{amsmath,amsthm,amssymb,graphicx,mathrsfs,tikz-cd}

\begingroup
    \makeatletter
    \@for\theoremstyle:=definition,remark,plain\do{%
        \expandafter\g@addto@macro\csname th@\theoremstyle\endcsname{%
            \addtolength\thm@preskip\parskip
            }%
        }
\endgroup
\usepackage{fullpage}
\usepackage{enumerate}
\usepackage{colonequals}
\usepackage{comment}
\usepackage[parfill]{parskip}
\usepackage{subcaption}
\usepackage{xcolor}
\usepackage{hyperref}

\theoremstyle{plain}
\newtheorem{thm}{Theorem}
\newtheorem{lemma}[thm]{Lemma}
\newtheorem{propn}[thm]{Proposition}
\newtheorem{cor}[thm]{Corollary}

\theoremstyle{definition}
\newtheorem{defn}[thm]{Definition}

\theoremstyle{remark}
\newtheorem{rmk}[thm]{Remark}

\newtheorem*{claim}{Claim}

\renewcommand{\subset}{\subseteq}
\newcommand{\R}{\mathbb{R}}
\newcommand{\Z}{\mathbb{Z}}
\newcommand{\C}{\mathbb{C}}

\newcommand{\B}{\mathbb{B}}
\newcommand{\K}{\mathbb{K}}

\newcommand{\A}{\mathcal{A}}
\newcommand{\F}{\mathcal{F}}

\newcommand{\T}{\mathcal{T}}

 \newcommand{\Index}{\mathrm{Index}}

\newcommand{\Ri}{\R \cup \{+\infty\}}

\begin{document}

\title{The smooth algebra of a one-dimensional singular foliation}
\author{
Michael Francis\\
Department of Mathematics\\ 
Pennsylvania State University\\ 
University Park, PA 16802, USA\\
\url{mjf5726@psu.edu}
}
\date{November 16, 2020}

\maketitle

\begin{abstract}
\noindent Androulidakis and Skandalis showed how to associate  a holonomy groupoid, a smooth convolution algebra and a C*-algebra to any singular foliation. In this note, we consider  the singular foliations of a one-dimensional manifold given by vector fields that vanish to   order $k$ at a point. We show that, whereas the C*-algebras of these foliations are divided into two isomorphism classes according to the parity of $k$, the smooth algebras  are pairwise nonisomorphic.  This is accomplished by analyzing certain natural ideals  in the smooth algebras. Issues of factorization with respect to convolution arise and are resolved using a  context-appropriate version of the   Diximier-Malliavin theorem. 
\end{abstract}

\section{Introduction}

Given any singular foliation $\F$  of a smooth manifold $M$, it was shown in \cite{AS[2007]} how to construct a holonomy groupoid $G(\F)$, a  smooth convolution algebra $\A(\F)$ and a  C*-algebra\footnote{As usual, more than one C*-completion can be considered,  including a reduced and a maximal version.  These differences will not matter here. Indeed, for the examples we consider, the natural representation of the maximal  C*-algebra on the direct sum of the $L^2$-spaces of the  leaves with trivial holonomy is  faithful, and so the two standard C*-completions agree with one another.} $C^*(\F)$. In this context, the word ``foliation'' is  understood in the following way:

\begin{defn}[\cite{AS[2007]}, Definitions~1.1,~1.2]
A \emph{foliation} $\F$ of a smooth manifold $M$ is a locally finitely-generated $C^\infty(M)$-module of compactly-supported vector fields on $M$  closed under Lie bracket. The \emph{tangent space} of $\F$ at $x \in M$ is $T_x\F \colonequals \{ X(x) : X \in \F\}$. The \emph{fiber} of $\F$ at $x$ is $A_x \F \colonequals \F/I_x \F$, where $I_x \subset C^\infty(M)$ denotes the ideal of functions vanishing at $x$. A foliation is called \emph{regular} if its tangent spaces have constant dimension and is called \emph{almost regular} if its fibers have constant dimension.  
\end{defn} 
The holonomy groupoid constructed in \cite{AS[2007]} can be quite poorly-behaved for general foliations. The case where $\F$ is almost regular is exactly the case where $G(\F)$ is a Lie groupoid. This case was previously treated by Debord in \cite{Debord}.  When $\F$ is almost regular, $\A(\F)$ is isomorphic to  $C_c^\infty(G(\F))$, the smooth convolution algebra of the  groupoid,   and $C^*(\F)$ is isomorphic to $C^*(G(\F))$, the C*-algebra of the groupoid. Here we are implicitly fixing a smooth Haar system on $G(\F)$ in order to make sense of convolution and bypass any discussion of densities. 

In this article, we are primarily concerned with the following specific family of  almost regular foliations (Example~1.3~(3), \cite{AS[2007]}).

\begin{defn}
For each positive integer $k$, we denote by $\F^k_\R$ the almost regular foliation of $\R$ singly-generated by $x^k \frac{d}{dx}$. That is, $\F^k_\R$ consists of all compactly-supported, smooth vector fields which vanish to order $k$ at $0$.  
\end{defn}

The holonomy groupoids of the foliations $\F^k_\R$ are not   difficult to understand; $G(\F^k_\R)$ is canonically isomorphic to the transformation groupoid $\R \rtimes_\phi \R$ associated to the flow of any complete vector field generating $\F^k_\R$. Accordingly, the smooth algebra $\A(\F^k_\R)$ is isomorphic to $C_c^\infty(\R \rtimes_\phi \R)$ and the C*-algebra $C^*(\F^k_\R)$ is isomorphic to  $C_0(\R) \rtimes_\phi \R$, the crossed product of the associated action of $\R$ on $C_0(\R)$. 

Our goal is to understand the extent to which the holonomy groupoids,  smooth algebras and C*-algebras of the foliations $\F^k_\R$ remember the positive integer $k$.   The most substantial part of this story is the analysis of the smooth algebras and the main result is the following:

\begin{thm}
The smooth convolution algebras of the foliations $\F^k_\R$, as $k$ ranges over the positive integers, are pairwise nonisomorphic.
\end{thm}

For C*-algebras, however, we have the following:

\begin{thm}
The C*-algebras of the foliations $\F^k_\R$, as $k$ ranges over the positive integers, are of two isomorphism types that are determined by the parity of $k$. 
\end{thm}

This demonstrates the principle that, even for singular foliations as simple as these ones, there can be information stored in the smooth algebra which is washed away when one passes to the C*-algebra.

The flow of a smooth vector field   with a finite order critical point at the origin, and no other critical points, is determined by  up to  topological conjugacy and time-reversal symmetry by the parity of the order of its zero. This remark already  explains   why there are at most two possibilities for the isomorphism-type of the C*-algebra $C_0(\R) \rtimes_\phi \R$. The fact that two different C*-algebras do indeed occur  follows from an index calculation.

The more substantial issue of showing that the smooth algebras $\A(\F^k_\R) \cong C_c^\infty(\R \rtimes_\phi \R)$ are pairwise nonisomorphic requires different methods. The intuition here is as follows: if the flow $\phi$ fixes the origin to $k$th order in $x$, where $x$ denotes the coordinate of the manifold $\R$,  then the convolution product on $C_c^\infty(\R \rtimes_\phi \R)$  is  ``commutative to $k$th order in $x$''. In more precise terms, our argument follows the following outline:

\begin{enumerate}
\item In each  of the convolution algebras $\A(\F^k_\R)$, consider the nested sequence of ideals ideals $x^p \cdot \A(\F^k_\R)$, $p$ a positive integer, consisting of functions vanishing to order $p$ on the isotropy group  of $0$. 
\item Show that the quotient $\frac{\A(\F^k_\R)}{x^p \cdot \A(\F^k_\R)}$ is commutative if and only if $p \leq k$.
\item Show that any isomorphism between $\A(\F^k_\R)$  and $\A(\F^\ell_\R)$ would  necessarily map the ideal sequence of the first algebra onto the ideal sequence of the second.
\end{enumerate}

Our approach to the third step of this program requires us to solve a problem concerning the existence of factorizations in the convolution algebra $C_c^\infty(\R \rtimes_\phi \R)$. To do this, we need to confirm that the following famous theorem of Dixmier-Malliavin remains  valid when $G$ is replaced by the transformation groupoid $\R \rtimes_\phi \R$. 

\begin{thm}[3.1 Th\'{e}or\`{e}me, \cite{Dixmier-Malliavin}]
If $G$ is any Lie group, then every $f \in C_c^\infty(G)$ can be expressed in the form $f=g_1*h_1 + \ldots g_N*h_N$, for some $N$, where $g_i,h_i \in C_c^\infty(G)$ and $*$ denotes convolution with respect to (a choice of) Haar measure. 
\end{thm}

In the separate article \cite{Francis[DM]}, we showed that Dixmier-Malliavin's theorem holds quite generally for any Lie groupoid $G$.

Let us now describe the organization of this paper. In Section~2, we summarize relevant aspects of Androulidakis-Skandalis's work and justify our use   of the models $G(\F^k_\R) \cong \R\rtimes_\phi \R$, $C_c^\infty(\R \rtimes_\phi \R)$ and $C_0(\R) \rtimes_\phi  \R$ for the groupoids, smooth convolution algebras and C*-algebras of the foliations $\F^k_\R$.

In   Section~3 we review the known fact that the so-called \emph{Weiner-Hopf extension}
\[ 0 \to \K \to C_0(\R \cup \{+\infty\}) \rtimes_\tau \R \to C^*(\R) \to 0 \]
arising from the translation action  on a one-sided extension of the line is isomorphic to the (nonunital) Toeplitz extension. One  standard approach is to use Laguerre functions to produce an operator of index 1 and appeal to Brown-Douglas-Fillmore theory (see \cite{Green} and \cite{Rieffel}). We will instead use a suitably unitarized form of the Cayley transform to give an explicit conjugacy relating the Weiner-Hopf extension to (an uncompressed form of) the Toeplitz extension. This approach appears to harken back to \cite{Devinatz}. 
 
 In Section~4 we use the results of Section~3 to analyze the C*-algebras $C^*(\F^k_\R)$ and obtain the above mentioned result that their isomorphism types are determined by the parity of $k$. For each $k$, there is a natural  representation  $\pi_k$ of  $C^*(\F^k_\R)$ on $L^2(\R_-) \oplus L^2(\R_+)$, where $\R_- \colonequals (-\infty,0)$ and $\R_+ \colonequals (0,\infty)$, the nonsingular leaves of $\F^k_\R$.  We furthermore show that each $\pi_k$ is faithful and that the image of $\pi_k$ does not equal the image of $\pi_\ell$ when $k \neq \ell$.
 
Finally, in Section~5, we prove pairwise   nonisomorphism of the smooth convolution algebras $\A(\F^k_\R)$ by examining their quotients by the sequence of ideals $x^p \cdot \A(\F^k_\R)$, as discussed above. We furthermore describe the quotient by the ideal $x^\infty   \cdot \A(\F^k_\R) \colonequals \bigcap_{p \geq 1} x^p  \cdot \A(\F^k_\R)$. This results in an interesting algebra of formal series whose indeterminate does not commute with its coefficients.
 
 \section{Preliminaries}\label{Preliminaries}

In this section we describe the models we will be using for the holonomy groupoids $G(\F^k_\R)$, smooth algebras $\A(\F^k_\R)$ and C*-algebras $C^*(\F^k_\R)$ of the foliations $\F^k_\R$, and briefly justify their usage. 

Given  a smooth action of a Lie groupoid $H$ on a smooth manifold $M$, we take the \emph{transformation groupoid} $M \rtimes H$ to be the Lie groupoid whose underlying manifold is  $M \times H$ with groupoid operations defined as follows:
\begin{align*}
\text{Source: } && d(x,h) &= x \\
\text{Target: } &&  r(x,h) &= h x \\
\text{Multiplication: } &&  (h_1x, h_2) (x,h_1) &= (x,h_1h_2)
\end{align*}

\begin{lemma}
Let $H$ be a connected Lie group acting on a smooth manifold $M$ and let $\F$ be the  singular foliation of $M$ determined by this action. If the set of $x \in M$ with trivial isotropy is dense in $M$, then the canonical map (see Example~3.4 in \cite{AS[2007]}) from the transformation groupoid $M \rtimes H$ to the holonomy groupoid $G(\F)$  is an isomorphism. 
\end{lemma}
\begin{proof}
According to Proposition~3.9 in \cite{AS[2007]}, we just need to check that $M \rtimes H$ is an $s$-connected \emph{quasigraphoid}, as defined at the reference. Since $H$ is connected, the transformation groupoid is $s$-connected. Suppose that $x \in M$ and $f$  is a smooth map from an open neighbourhood $U$ of $x$ to $H$ which satisfies $f(y)y=y$ for all $y \in U$. Then, $f(y)=1$ for a dense set of $y$ in $U$ by hypothesis, and so $f$ is identically equal to $1$. This shows that  the only local bisections which induce the identity mapping are the trivial ones, i.e. $M \rtimes H$ is a quasigraphoid.
\end{proof}

 \begin{cor}\label{finitecri}
 Let $X$ be a complete vector field on a smooth manifold $M$. Let $\F$ be the singular foliation singly-generated by $X$. If $X$ has finitely many critical points and periodic orbits, then the transformation groupoid $M \rtimes_X \R$ determined by the flow of $X$ is canonically isomorphic to the holonomy groupoid $G(\F)$. 
 \end{cor}
 
 By the corollary above one has a canonical isomorphism 
 \begin{align}
 G(\F^k_\R) \cong \R \rtimes_\phi \R 
\end{align}
where $\phi$ is the flow of any complete vector field  generating $\F^k_\R$. The most obvious generator for the singular foliation $\F^k_\R$ is $x^k \frac{d}{dx}$, but note that this vector field is not complete when $k \geq 2$. One can, however, rescale it in order to get a complete vector field $X_k$ generating the same foliation. For example, one could  use $X_k =(1+x^2)^{-\frac{k-1}{2}}  x^k\frac{d}{dx}$ which resembles $x^k \frac{d}{dx}$ near the origin, but has sublinear growth. 
\begin{rmk}
In fact, it still makes perfect sense to talk about the transformation groupoid $M  \rtimes_X \R$ associated to smooth vector field even when $X$ is not complete. One just takes the underlying manifold of the groupoid to instead be the domain of the flow, an open subset of $\R \times M$, and defines the groupoid operations exactly as in the complete case. We could therefore get away with using the vector fields $x^k \frac{d}{dx}$ themselves rather than rescalings thereof.
\end{rmk}

The transformation groupoid $\R \rtimes_\phi \R$ has a natural right Haar system given by copying the usual Lebesgue measure on each source fiber. The convolution product determined on $C_c^\infty(\R \rtimes_\phi \R)$ by this Haar system is as follows:  
\begin{align}
(f  * g)(x,t) = \int_\R f(\phi_s(x),t-s,)g(x,s) \ ds.
\end{align}
Note that changing the smooth Haar system results in a canonically isomorphic product. We also have an adjoint operation  on  $C_c^\infty(\R \rtimes_\phi \R)$ given by
\[ f^*(x,t) = \overline{ f(\phi_t(x),-t)}. \]
The $*$-algebra structure on $C_c^\infty(\R \rtimes_\phi \R)$ defined by these operations is canonically isomorphic to the smooth convolution algebra $\A(\F^k_\R)$ defined  in Section~4.3 of \cite{AS[2007]} (one needs to translate from the language of densities to the language of functions).

One also has that the C*-algebra  $C^*(\F^k_\R)$ is canonically isomorphic to the crossed product C*-algebra $C_0(\R) \rtimes_\phi \R$, where $\phi$ denotes the flow of any complete vector field generating $\F^k_\R$ as well as the corresponding action of $\R$ on $C_0(\R)$ defined by $(\phi_t f)(x) = f(\phi_{-t}(x))$. Let us briefly explain why this is so.  Firstly, it is well-known that there is a canonical isomorphism  $\Psi  : C^*(\R \rtimes_\phi \R) \to C_0(\R) \rtimes_\phi \R$ from the usual groupoid C*-algebra, in the sense of \cite{Renault[BOOK]},   to the crossed product C*-algebra. The groupoid C*-algebra  is the completion of $C_c^\infty(\R \rtimes_\phi \R)$ with respect to the largest C*-norm dominated by the $L^1$-norm
\begin{align}\label{grpd1}
 \|f\|_{1,\text{Groupoid}} = \sup_{x \in \R} \left\{ \int_\R |f(x,t) |\ dt, \int_\R |f^*(x,t)| \ dt \right\}. 
 \end{align}
 The isomorphism $\Psi$ is determined by its restriction to $C_c^\infty(\R \rtimes_\phi \R)$ which has image contained in the dense subalgebra $C_c(\R, C_0(\R)) \subset C_0(\R) \rtimes_\phi \R$ and is given by 
\[ (( \Psi f)(t))(x) = f(\phi_{-t}(x),t). \]
In Section 4.4 of \cite{AS[2007]}, the C*-algebra $C^*(\F^k_\R)$ is defined as the completion of $\A(\F^k_\R)$ with respect to the largest C*-norm dominated by an $L^1$-norm  $\| \cdot \|_{1,AS}$ defined by a choice of Riemannian metric\footnote{The choice of metric does not affect  the completion.} on the base manifold (\cite{AS[2007]}, Definition~4.8).   Choosing the standard Riemann metric on $\R$, one may check that Androulidakis-Skandalis's $L^1$-norm becomes the norm on $C_c^\infty(\R \rtimes_\phi \R)$ defined by 
\begin{align}\label{AS1}
\|f\|_{1,\text{AS}} = \sup_{x \in \R} \left\{ \int_\R |f(x,t) | \beta(x,t) \ dt, \int_\R |f^*(x,t)| \beta(x,t) \ dt \right\}
\end{align}
where
 \[ \beta(x,t) = \begin{cases}
 (\phi_t'(x))^{1/2} & \text{ if } x \neq 0 \\
 1 & \text{ if } x =0. \end{cases}  \]
 One may wonder whether these (different) $L^1$-norms  \eqref{grpd1} and \eqref{AS1} really determine the same C*-completions. This is so, and follows from the fact that the map $\beta : \R \rtimes_\phi \R \to (0,\infty)$ appearing in \eqref{AS1} is a 1-cocycle, i.e. satisfies $\beta( \gamma_1 \gamma_2) = \beta(\gamma_1) \beta(\gamma_2)$, where $\gamma_1,\gamma_2 \in \R \rtimes_\phi \R$. One has that the Hilbert space representations dominated by $\|\cdot \|_{1,\text{Groupoid}}$ are the same as the ones obtained by integrating representations of the underlying groupoid. This is shown in \cite{Renault}, 1.7~Proposition.  Following the same argument, one may check that the same remains true when  $\|\cdot\|_{1,\text{Groupoid}}$ is adjusted by a 1-cocycle; the 1-cocycle drops out of the main estimate given on pp.~50 of \cite{Renault}.

\section{The Weiner-Hopf extension}\label{WHsection}

In this section we explain  how a suitably unitarized form of the Cayley transform can be used to  relate the  Weiner-Hopf extension (defined below)  to  the more familiar Toeplitz extension.  The precise statements are given in Theorem~\ref{mainext} and Corollary~\ref{mainiso}.

Form a one-sided extension $\Ri$ of the  real line by adjoining  a positive infinity, but not a negative infinity. Let $M$ denote the multiplication representation of $C_0(\Ri)$ on $L^2(\R)$. Let $\lambda$ denote the regular representation of $\R$ on $L^2(\R)$ as well as its integrated form,   a representation of $C^*(\R)$. On $L^1(\R) \subset C^*(\R)$, this is the usual convolution representation $\lambda(f) \xi =  f * \xi$. 

\begin{defn}
The \emph{Weiner-Hopf} algebra is the C*-algebra $\T_\R$ on $L^2(\R)$ generated by products $M(f)\lambda(g)$ where $f \in C_0(\Ri)$ and $g \in C^*(\R)$. There is a $*$-homomorphism $\sigma_\R : \T_\R \to  C^*(\R)$ called the \emph{symbol map},  determined by $\sigma_\R( M(f) \lambda(g)) = f(+\infty) g$ whose kernel is $\K(L^2(\R))$, the ideal of compact operators. The resulting extension
\begin{align*}
\begin{tikzcd}[ampersand replacement=\&]
 0 \ar[r]  \& \K(L^2(\R)) \ar[r] \&  \T_\R \ar[r,"{\sigma_\R}"] \&  C^*(\R) \ar[r] \& 0 
 \end{tikzcd}
 \end{align*}
will be referred to as the \emph{Weiner-Hopf extension}. 
\end{defn}

That the symbol map $\sigma_\R$  is defined at all will be a consequence of the discussion to follow, in which the Weiner-Hopf algebra will be related to the Toeplitz algebra. For the moment, let us assume $\sigma_\R$ exists and obtain a crossed product description of the Weiner-Hopf algebra.

Let $\R$ act  on the extended line $\Ri$ by translation, fixing the point $+\infty$. Let  $\tau$ be the corresponding action of $\R$ on $C_0(\Ri)$ defined by
\begin{align*}
(\tau_tf)(x) = f(x-t). 
\end{align*}
 Evaluation at  $+\infty$  yields an $\R$-equivariant exact sequence of commutative C*-algebras:
\begin{align*}
\begin{tikzcd}[ampersand replacement=\&] 
0 \ar{r} \& C_0(\R) \ar{r} \& C_0(\Ri) \ar[r]  \&  \C \ar{r} \& 0.
\end{tikzcd}
\end{align*}
Taking the crossed product of the above sequence by $\R$, we obtain another exact sequence of C*-algebras.
\begin{align*}
\begin{tikzcd}[ampersand replacement=\&] 
0 \ar{r} \& C_0(\R)  \rtimes_\tau \R \ar{r} \& C_0(\Ri) \rtimes_\tau \R \ar[r]  \&  C^*(\R) \ar{r} \& 0.
\end{tikzcd}
\end{align*}
This sequence is of some importance  in noncommutative geometry. For example, in   \cite{Rieffel},  it is used  to  realize the Connes-Thom isomorphism of \cite{Connes[Thom]} as a boundary map in K-theory.

Recall that a (maximal) crossed product C*-algebra $A \rtimes_\alpha G$ contains canonical copies of $A$ and $C^*(G)$ in its multiplier algebra and that $A \rtimes_\alpha G$ is generated by  elementary products  $a \times b$, where $a \in A$,  $b \in C^*(G)$. Each  nondegenerate  representation  $\pi$ of $A \rtimes_\alpha G$ arises uniquely from a \emph{covariant pair} $(\pi_A, \pi_G)$ where $\pi_A$ is a representation of $A$, $\pi_G$ is a unitary representation of $G$ and $\pi_G(t) \pi_A(a) \pi_G(t)^{-1} = \pi_A( \alpha_t(a))$ holds for all $a \in A$, $t \in G$. The corresponding representation $\pi$ of $A \rtimes_\alpha G$ is determined by $\pi(a \times b) = \pi_A(a) \pi_G(b)$ for all $a \in A$, $b \in C^*(G)$. Here we abuse notation, denoting the integrated form of $\pi_G$, a representation of $C^*(G)$, by the same symbol as its inducing unitary representation. For more information on crossed products, one may refer to the extensive survey \cite{Williams}.

There is a natural representation  $\pi : C_0(\Ri) \rtimes_\tau \R \to \mathbb{B}(L^2(\R))$ coming from the covariant pair $(M,\lambda)$  where, as above, $M$ denotes the multiplication  representation and $\lambda$ denotes  the regular representation of $\R$.  The representation  $\pi$ is thus determined  on elementary products by
\begin{align}\label{pi}
\pi(f \times g) = M(f) \lambda(g)  &&  f \in C_0(\Ri),   g \in C^*(\R).
\end{align}
By construction, the image of the representation $\pi$ is the Weiner-Hopf algebra $\T_\R$. 
It is well-known that the restriction of $\pi$ to the ideal $C_0(\R) \rtimes_\tau \R$ is a faithful representation onto the C*-algebra $\K(L^2(\R))$ of compact operators. This is  a C*-algebraic formulation of the \emph{Stone-von Neumann theorem}, describing  the  representation theory of the canonical commutation relations. Actually, one  has more generally that  $C_0(G) \rtimes G \cong \K(L^2(G))$    for any locally compact group $G$. One can then see that $\pi$ is also faithful on $C_0(\Ri) \rtimes_\tau \R$ by consideration of the following commutative diagram:
\begin{align*}
\begin{tikzcd}[ampersand replacement=\&] 
0 \ar[r] \& C_0(\R)\rtimes_\tau \R \ar[r] \ar[d,"\pi"]\& C_0(\Ri)\rtimes_\tau\R \ar[d,"\pi"] \ar[r] \& C^*(\R) \ar[r] \ar[d,equals]\& 0 \\
0 \ar[r] \& \K(L^2(\R)) \ar[r] \& \T_\R \ar[r,"{\sigma_\R}"] \& C^*(\R) \ar[r] \& 0.
\end{tikzcd}
\end{align*}
The vertical maps on the right and the left are isomorphisms, whence the one in the middle is as well.

We shall review the known fact that the Weiner-Hopf extension is isomorphic to  the (nonunital form of the) Toeplitz extension. In \cite{Rieffel} as well as  \cite{Green} (Lemma~6),  this is shown by using Laguerre functions to produce an operator of index 1 and appealing to Brown-Douglas-Fillmore theory. However,  we feel that a clearer understanding of the Weiner-Hopf  extension is reached by using  the Cayley and Fourier transform in concert to explicitly relate it to the Toeplitz extension. This observation is not new, but  deserves to be better known. The approach by way of the Cayley transform seems to originate from  \cite{Devinatz} (see pp.~82-83 therein). See also the historical remark on page 2 of \cite{Douglas} as well as  Proposition~3.7.1 of \cite{Torpe}.

Let $H^2(S^1)$ denote the \emph{Hardy space} of the circle,  the closed subspace of $L^2(S^1)$ spanned by the basis vectors $\frac{1}{\sqrt{2 \pi}} z^n$ where $n \geq 0$.  One may also characterize the elements  of $H^2(S^1)$ as the boundary values of a  corresponding space of holomorphic functions on the open disk.  Recall that the multiplication representation $M :C^1(S^1) \to \B(L^2(S^1))$ commutes modulo compact operators with $P_{H^2(S^1)}$, the projection onto the Hardy space. The \emph{Toeplitz operator} with symbol $f \in C(S^1)$ is defined as $T_f \colonequals P_{H^2(S^1)}M(f)P_{H^2(S^1)}$, the compression of  $M(f)$ to the Hardy space. The \emph{Toeplitz algebra} is the C*-algebra $\T$ on $H^2(S^1)$ generated by the Toeplitz operators. In fact,  $\T$ is singly-generated by the unilateral shift operator $T_z$.  The Toeplitz algebra sits in a canonical extension
\begin{align*}
\begin{tikzcd}[ampersand replacement=\&]
0 \ar[r] \& \K(H^2(S^1)) \ar[r] \& \T  \ar[r,"\sigma"] \& C(S^1) \ar[r] \& 0 
\end{tikzcd}
\end{align*}
called the \emph{Toeplitz extension}. The homomorphism $\sigma$ is called the \emph{symbol map} and satisfies $\sigma(T_f)=f$ for all $f \in C(S^1)$. Indeed, the assignment $f \mapsto T_f$ is a completely positive splitting for $\sigma$.   
One may also consider the \emph{nonunital Toeplitz algebra} $\T_0$ defined as the preimage by $\sigma$ of the codimension one ideal $C_0(S^1) \subset C(S^1)$ consisting of functions vanishing at $1 \in S^1$. It is singly-generated by $T_{1-z}$ and fits into the  \emph{nonunital Toeplitz extension}:
\begin{align*}
\begin{tikzcd}[ampersand replacement=\&]
0 \ar[r] \& \K(H^2(S^1)) \ar[r] \& \T_0  \ar[r,"\sigma"] \& C_0(S^1) \ar[r] \& 0.
\end{tikzcd}
\end{align*}
It is sometimes desirable to not compress to the Hardy space and work instead with
\begin{align}\label{Tbar}
\overline \T_0 \colonequals \T_0 + \K(L^2(S^1)).
\end{align}
There is a unique extension $\overline \sigma$ of the symbol map to $\overline \T$ which has $\ker(\overline \sigma) = \K(L^2(S^1))$. From these definitions, we have the following commutative diagram in which the vertical maps are inclusions (or, in the case of the rightmost map, an equality).
\begin{align*}
\begin{tikzcd}[ampersand replacement=\&] 
0 \ar{r} \& \K(H^2(S^1))  \ar{r} \ar[d,hookrightarrow] \& \T_0 \ar[d,hookrightarrow] \ar[r,"{\sigma}"]  \&  C_0(S^1) \ar{r} \ar[d,equals] \& 0 \\
0 \ar{r} \& \K(L^2(S^1))  \ar{r} \& \overline{\T_0}  \ar[r,"{\overline \sigma}"]  \&  C_0(S^1) \ar[r]   \& 0  
\end{tikzcd} 
\end{align*}
A convenient feature of this uncompressed form of the (nonunital) Toeplitz extension is that all three terms in the exact sequence on the bottom are represented on the same Hilbert space $L^2(S^1)$.  Brown-Douglas-Fillmore theory gives, however, that the extensions on the top and bottom of the above diagram are in fact the same. 
\begin{lemma}\label{topiso}
There exists an isometric isomorphism $V : H^2(S^1) \to L^2(S^1)$ such that $\mathrm{Ad}_V$ carries $\T_0$ onto $\overline \T_0$ and the following diagram is commutative:
\begin{align*}
\begin{tikzcd}[ampersand replacement=\&] 
0 \ar{r} \& \K(H^2(S^1))  \ar{r} \ar[d,"{\mathrm{Ad}_V}"] \& \T_0 \ar[d,"{\mathrm{Ad}_V}"] \ar[r,"\sigma"]  \&  C_0(S^1) \ar{r} \ar[d,equals] \& 0 \\
0 \ar{r} \& \K(L^2(S^1))  \ar{r} \& \overline{\T_0}  \ar[r,"{\overline \sigma}"]  \&  C_0(S^1) \ar[r]   \& 0. 
\end{tikzcd} 
\end{align*}
\end{lemma}
\begin{proof}
Let $S=T_z$, the unilateral shift operator on $H^2(S^1)$, and let $\overline S$ be the direct sum of $S$ with the identity operator on the orthogonal complement of $H^2(S^1)$ in $L^2(S^1)$. Since $S$ and $\overline S$ are both isometries of index $-1$, the classification of essentially normal operators implies the existence of an isometric isomorphism $V : H^2(S^1) \to L^2(S^1)$ such that $VS V^{-1}$ equals $\overline S$ modulo compact operators\footnote{Obviously $S$ and $\overline S$ are not actually conjugate; $S$ has no eigenvectors, whereas $\overline S$ acts as  the identity on an infinite-dimensional subspace.}. Thus, $VT_{1-z}V^{-1}$ equals $T_{1-z}$, modulo compact operators. Recalling that $\T_0$ is singly-generated by $T_{1-z}$ and contains $\K(H^2(S^1))$, the result follows. 
\end{proof}

Recall that one similarly has a \emph{Hardy space} $H^2(\R) \subset L^2(\R)$ for the real line. As in the case of the circle, $H^2(\R)$ admits two descriptions: a harmonic analysis description as the space of $\xi \in L^2(\R)$ whose Fourier transform  has support contained in $[0,\infty)$, and a complex function theory description as the space of boundary values of a corresponding space of holomorphic functions on the upper half-plane. Recall the \emph{Cayley transform}
\begin{align} 
w : \R \to S^1 \setminus\{1\} && w(t) = \frac{t-i}{t+i}.
\end{align}
Like any other diffeomorphism  $\R \to S^1 \setminus \{1\}$, the Cayley transform   determines an isometric isomorphism $L^2(S^1) \to L^2(\R)$ given by  $\xi \mapsto  |w'|^{1/2} (\xi \circ w)$. Recall however that, viewed as a Mob\"{i}us transformation, $w$ carries the upper half plane to the unit disk. Furthermore, its derivative   $w'(t) = \frac{2i}{(t+i)^2}$ is the square of a function holomorphic in the upper half-plane. One is therefore encouraged to adjust the phase of the na\"{i}ve identification $L^2(S^1) \to L^2(\R)$ given above and instead use the isometric isomorphism $W$ defined by the following  meromorphic formula:
\begin{align}\label{cayley}
W : L^2(S^1) \to L^2(\R) &&  (W\xi)(t) =  \frac{\sqrt{2}}{t+i} \xi(w(t)).
\end{align}
Although this adjustment to the phase might seem an inconsequential matter, it in fact leads to $W$ having the highly useful property of mapping $H^2(S^1)$ onto $H^2(\R)$. Indeed, 
\begin{align*} 
\begin{tikzcd}[ampersand replacement=\&]
\left\{  \tfrac{1}{\sqrt{2\pi}} z^n  : n \geq 0 \right\} \ar[r,"W"] \&  \left\{ \frac{1}{\sqrt{\pi}} \left(\frac{t-i}{t+i}\right)^n \frac{1}{t+i} : n \geq 0  \right\} \\
\left\{   \tfrac{1}{\sqrt{2\pi}} z^n  : n < 0 \right\} \ar[r,"W"] \& \left\{ \frac{1}{\sqrt{\pi}} \left(\frac{t+i}{t-i}\right)^n \frac{1}{t-i}:  n \geq 0 \right\} . 
\end{tikzcd}
 \end{align*}

We also have that $W$  intertwines the multiplication representation of $C_0(\R)$ on $L^2(\R)$ with the multiplication representation of $C_0(S^1)$ on $L^2(S^1)$ in the expected way, because the phase adjustment is itself implemented by a (circle-valued) multiplication operator. We summarize the properties of $W$ just discussed in the following lemma.

\begin{lemma}\label{cayleyprop}
Define $W : L^2(S^1)\to L^2(\R)$ by $(W\xi)(t) = \frac{\sqrt{2}}{t+i}\xi \left(\frac{t-i}{t+i}\right)$. Then:
\begin{enumerate}
\item $W$ is an isometric isomorphism $L^2(S^1) \to L^2(\R)$,
\item $W$ maps $H^2(S^1)$ onto $H^2(\R)$,
\item $W$ satisfies $WM(f)W^{-1} = M(f \circ w)$ for all $f \in C_0(S^1)$, where $w(t) = \frac{t-i}{t+i}$ is the Cayley transform.
\end{enumerate}
\end{lemma}

We use the following convention for the Fourier transform:
\begin{align}\label{fourier}
\widehat f (s) = \int_\R f(t) e^{-2 \pi i st} \ dt && f  \in C_c(\R).
\end{align}
The Fourier transform  extends uniquely to a C*-algebra isomorphism, written  simply as
\begin{align*} f \mapsto \widehat f : C_0(\R) \to C^*(\R).
\end{align*}
The Fourier transform also extends uniquely to an isometric isomorphism of Hilbert spaces, which we write as 
\begin{align*}
U  : L^2(\R) \to L^2(\R)
\end{align*}
in order to notationally distinguish these slightly different concepts. The key properties of the Fourier transform needed are gathered below. 

\begin{lemma}\label{fourierprop}
Let $U : L^2(\R) \to L^2(\R)$ be the Fourier transform defined by $(U\xi)(s) = \int_\R \xi(t) e^{-2\pi i st} \ dt$ for $\xi \in L^1(\R) \cap L^2(\R)$ and extended by continuity. Then:
\begin{enumerate}
\item $U$ is an isometric isomorphism $L^2(\R) \to L^2(\R)$,
\item $U$ maps $H^2(\R)$ onto $L^2([0,\infty))$,
\item $U$ satisfies $UM(f)U^{-1} = \lambda(\widehat f)$ for all $f \in C_0(\R)$, where $f \mapsto \widehat f$ is the Fourier isomorphism $C_0(\R) \to C^*(\R)$ and $\lambda$ is the regular (i.e. convolution) representation of $C^*(\R)$ on $L^2(\R)$. 
\end{enumerate}
\end{lemma}

We now present the conjugacy between the Weiner-Hopf and Toeplitz extensions. 

\begin{thm}\label{mainext}
Let $W : L^2(S^1) \to L^2(\R)$ and $U : L^2(\R) \to L^2(\R)$ be the Hilbert space isomorphisms associated to the Cayley and Fourier transforms
Then, $\mathrm{Ad}_U \circ \mathrm{Ad}_W$ carries $\overline \T_0$, the Toeplitz algebra in its uncompressed, nonunital form
onto the Weiner-Hopf algebra $\T_\R$. Furthermore, the map $\sigma_\R$ defined by commutativity of the following diagram 
\begin{align*}
\begin{tikzcd}[ampersand replacement=\&] 
0 \ar[r] \& \K(L^2(S^1)) \ar[r] \ar[d,"{\mathrm{Ad}_U \circ \mathrm{Ad}_W}"]\& \overline \T_0 \ar[d,"{\mathrm{Ad}_U \circ \mathrm{Ad}_W}"] \ar[r,"{\overline \sigma}"] \& C_0(S^1) \ar[r] \ar[d,"{f \mapsto \widehat{f \circ w}}"]\& 0 \\
0 \ar[r] \& \K(L^2(\R)) \ar[r] \& \T_\R  \ar[r,"{\sigma_\R}"] \& C^*(\R) \ar[r] \& 0
\end{tikzcd}
\end{align*}
satisfies $\sigma_\R(M(f) \lambda(g)) = f(+\infty) g$ for all $f \in C_0(\Ri)$, $g \in C^*(\R)$. 
\end{thm}

\begin{proof}
By definition, $\T_\R$ is generated by products $M(f)\lambda(g)$ where $f \in C_0(\Ri)$ and $g \in C^*(\R)$. By definition, $\overline \T_0$ is generated by products $P_{H^2(S^1)} M_f P_{H^2(S^1)}$, $f \in C_0(S^1)$, together with $\K(L^2(S^1))$. Let  $\chi$ denote the characteristic function of the positive half-line $[0,\infty)$, so that $M(\chi)$ is the orthogonal projection onto $L^2([0,\infty)) \subset L^2(\R)$. The image of $\overline \T_0$ by $\mathrm{Ad}_U \circ \mathrm{Ad}_W$ is then generated by products $M(\chi)\lambda(g)M(\chi)$, $g \in C^*(\R)$ together with $\K(L^2(\R))$ (see Lemmas~\ref{cayleyprop} and \ref{fourierprop}). Since $\T_\R$ contains the compacts, the claim that $\mathrm{Ad}_U \circ \mathrm{Ad}_W$ maps $\overline \T_0$ onto $\T_\R$ will be proven once we establish the following:
\begin{claim}
If $f \in C_0(\Ri)$, $g \in C^*(\R)$ and $f(+\infty)=1$, then $M(f)\lambda(g)$ is equal to $M(\chi) \lambda(g) M(\chi)$, modulo compact operators.
\end{claim}

Firstly, since $M(C_0(\R)) \lambda(C^*(\R)) \subset \K(L^2(\R))$, it does no harm to assume that $f$ is identically equal to $1$ on $[0,\infty)$. We can then write $f-\chi = (1-\chi)h$ where $h \in C_0(\R)$ and therefore, using that $M(h)\lambda(g)$ is compact and that compact operators form an ideal, conclude that $M(f)\lambda(g)$ is equal to $M(\chi) \lambda(g)$ modulo compact operators.  Secondly, recall that the Hardy projection $P_{H^2(S^1)}$ commutes, modulo compact operators, with the multiplication representation of $C(S^1)$ on $L^2(S^1)$. Conjugating through $U \circ W : L^2(S^1) \to L^2(\R)$, this says that the projection $M(\chi)$ commutes modulo compact operators with the regular representation $\lambda$ of $C^*(\R)$ on $L^2(\R)$. Therefore, $M(\chi) \lambda(g) =M(\chi) \lambda(g) M(\chi)$, modulo compact operators, establishing the claim.

The assertion about $\sigma_\R$ also follows from the above claim.

\end{proof}

Since the uncompressed form of the nonunital Toeplitz extension is isomorphic to the usual nonunital Toeplitz extension, we therefore have

\begin{cor}\label{mainiso}
The Weiner-Hopf algebra $\T_\R$ is isomorphic to the nonunital Toeplitz algebra $\T_0$.
\end{cor}
\begin{proof}
Follows from Theorem~\ref{mainext} and Lemma~\ref{topiso}.
\end{proof}

\begin{cor}\label{genX}
The K-theory boundary map $K_1(C^*(\R)) \to K_0(\K(L^2(\R))) \cong \Z$ induced by the Weiner-Hopf extension $0 \to \K(L^2(\R)) \to \T_\R \overset{\sigma_\R}{\to} C^*(\R) \to 0$ sends $[1-b] \mapsto -1$, where $b \in L^1(\R) \subset C^*(\R)$ is given by
\begin{align*}
b(t)  = \begin{cases} e^{-t/2} & t \geq 0 \\
0 & t<0. \end{cases}
\end{align*} 
\end{cor}
\begin{proof}
This follows from the corresponding calculation for the Toeplitz extension, and naturality of the K-theory boundary map. Note that $[1-b]$ is chosen to be the image under the Fourier isomorphism $C_0(\R) \to C^*(\R)$ of the usual generator of $K_1(C_0(\R))$ given as the class of (any) loop of winding number $1$. 
\end{proof}

 \begin{rmk}
Obviously one can combine the Fourier transform of the line, the Cayley transform and the Fourier transform of the circle to realize $C^*(\R)$ as a codimension-1 ideal in $C^*(\Z)$. One may understand this section as explaining how, by appropriately unitarizing the Cayley transform, one can promote this observation to   the level of extensions.
\begin{align*}
\begin{tikzcd}[ampersand replacement=\&] 
0 \ar[r] \& \K(L^2(\R)) \ar[r]  \ar[d] \& C_0(\R\cup\{+\infty\}) \rtimes \R  \ar[r] \ar[d] \& C^*(\R)  \ar[r] \ar[d]\& 0 \\
0 \ar[r] \& \K(L^2(\Z)) \ar[r]  \& C_0(\Z\cup\{+\infty\}) \rtimes \Z \ar[r] \& C^*(\Z) \ar[r]  \& 0 
\end{tikzcd}
\end{align*}
\end{rmk}

\section{The C*-algebras of the foliations $\F^k_\R$}\label{Cstarclass}

In this section, we classify the C*-algebras of the foliations $\F^k_\R$ up to isomorphism. It is quite straightforward to see that all the  $C^*(\F^k_\R)$ with $k$ even are  isomorphic, and that all the $C^*(\F^k_\R)$ with $k$ odd are  isomorphic. The main task is therefore to distinguish these two cases by some invariant. For this purpose, we use that $C^*(\F^k_\R)$ contains a unique essential ideal $I$ isomorphic to $\K \oplus \K$. The ideal $I$ determines a map $K_1(C^*(\F^k_\R)/I) \to \Z$ given as the composition of  the $K$-theory boundary map $K_1(C^*(\F^k_\R)/I) \to K_0(I)$ with  the addition map $K_0(I) \cong \Z \oplus \Z \to \Z$. We show this map $K_1(C^*(\F^k_\R)/I) \to \Z$   is zero when $k$ is even and an isomorphism onto $2\Z$ when $k$ is odd.

As in the case of regular foliations,  the C*-algebra $C^*(\F)$ of a singular foliation $\F$ is a naturally represented on the $L^2$ space of any of its leaves  (or the holonomy cover of that leaf, as appropriate). See Section~4.5 in \cite{AS[2007]}. In the case of $\F^k_\R$, the leaves with no holonomy consist of the two half lines $\R_- \colonequals (-\infty,0)$ and $\R_+ \colonequals (0,\infty)$. We show that the natural representation of $C^*(\F^k_\R)$ on $L^2(\R_-) \oplus L^2(\R_+)$ is faithful and that the images of these representations as $k$ varies are distinct C*-algebras (even though there are only of two isomorphism types). This also shows that the reduced and maximal C*-algebras are the same for these foliations.

As discussed in Section~\ref{Preliminaries}, the holonomy groupoid of  $\F^k_\R$ is isomorphic to the transformation groupoid $\R \rtimes_\phi \R$ associated to the flow of any  complete vector field generating $\F^k_\R$. The $C*$-algebra of $\F^k_\R$ is then isomorphic to the crossed product $C_0(\R) \rtimes_\phi \R$ of the associated action on $\R$.  Throughout this section, $\phi$ will denote a smooth action of $\R$  on $\R$ which fixes the origin and is transitive on each of the half lines $\R_- \colonequals (-\infty,0)$ and $\R_+\colonequals (0,\infty)$. We also use $\phi$ to denote the corresponding action of $\R$ on  $C_0(\R)$ by $*$-automorphisms: 
\begin{align}\label{C0action}
 (\phi_tf)(x) =  f(\phi_{-t}(x)).
 \end{align}
 We want to understand the  structure of the crossed product C*-algebra $C_0(\R) \rtimes_\phi \R$. Working up to orientation-preserving topological conjugacy, there are four types of flows to consider.  We  assign each case  a ``bi-index'', as tabulated below. The first component indicates whether the origin is a source or a sink for the left half line and the second component indicates whether the origin is a source or a sink for the right half line. 
\begin{figure}[htbp]
\centering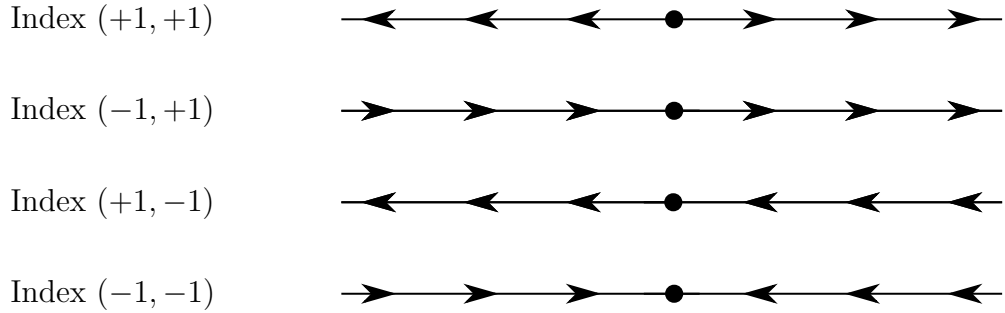
\caption{Indexing convention for flows on $\R$ with unique fixed point $0$.}\label{biindex}
\end{figure}

\begin{rmk}
If orientation-reversing topological conjugacies are  also allowed, we are reduced to three cases; the index $(-1,1)$ case becomes equivalent to the index $(1,-1)$ case. If time-reversal symmetry is allowed, we are reduced to two cases; the index $(-1,1)$ and $(1,-1)$ cases become equivalent and the index $(1,1)$ and $(-1,-1)$ cases become equivalent. Such equivalences do not change the C*-algebra, so it follows that there are at most two possibilities for the isomorphism type of $C_0(\R) \rtimes_\phi \R$. That two types do in fact occur will follow from an index calculation.  
\end{rmk}
Evaluation at $0$ yields an $\R$-equivariant exact sequence of commutative C*-algebras
\begin{align*}
\begin{tikzcd}[ampersand replacement=\&] 
0 \ar{r} \& C_0( \R_-) \oplus C_0(\R_+)  \ar{r} \& C_0(\R) \ar[r]  \&  \C \ar{r} \& 0,
\end{tikzcd} 
\end{align*}
where $\R_- \colonequals (-\infty,0)$, $\R_+ \colonequals (0,\infty)$.  The restriction of the action to either of the half lines is conjugate to  the translation action of $\R$ on itself. Thus, taking the crossed product by $\R$  yields that $C_0(\R) \rtimes_\phi \R$ sits in an exact sequence of the form
\begin{align}\label{inducer}
\begin{tikzcd}[ampersand replacement=\&] 
0 \ar[r] \&  \K \oplus \K \ar[r] \& C_0(\R) \rtimes_\phi \R \ar{r} \& C^*(\R) \ar{r} \& 0 .
\end{tikzcd} 
\end{align}
First we show that there are four possibilities for the isomorphism type of this natural extension if we slightly alter the usual definition of isomorphism of extensions by insisting that the order of the factors $\K \oplus \K$ is preserved.

Let $M$ denote the multiplication representation of $C_0(\R)$ on $L^2(\R_-) \oplus L^2(\R_+)$. Let $U$ denote the unitary representation of $\R$ on $L^2(\R_-) \oplus L^2(\R_+)$ given by $(U_t \xi)(x) = (\phi_{-t}'(x))^{1/2} \xi(\phi_t(x))$. Then $(M,U)$ is a covariant pair and determines a representation  $\pi$ of $C_0(\R) \rtimes_\phi \R$ on $L^2(\R_-) \oplus L^2(\R_+)$. Observe that the restriction of $\pi$ to $C_0(\R_- \cup \R_+) \rtimes_\phi \R$ is an isomorphism onto $\K(L^2(\R_-)) \oplus \K(L^2(\R_+))$.

\begin{thm}
Let $\phi$ be  a smooth action  of $\R$ on $\R$ with $0$ as its unique fixed point. Then:
\begin{enumerate}
\item The representation $\pi$ of $C_0(\R) \rtimes_\phi \R$ on $L^2(\R_-) \oplus L^2(\R_+)$ defined above is faithful. 
\item The  K-theory boundary map $K_1(C^*(\R)) \to \Z \oplus \Z$ arising from the exact sequence \eqref{inducer} sends the generator $[1-b]$  from Corollary~\ref{genX}   to the index of the flow $\phi$, as tabulated in Figure~\ref{biindex}.
\end{enumerate}
\end{thm}
\begin{proof}
We define $\T_\R^{-1}$  to be $\T_\R$ equipped with its usual symbol map $\sigma_\R^{-1} \colonequals \sigma_\R$. We define $\T_\R^{+1}$ to be $\T_\R$ equipped with the time-reversed symbol map $\sigma_\R^{+1}\colonequals r \circ \sigma_\R$, where $r : C^*(\R) \to C^*(\R)$ is determined on $L^1(\R)$ by $(rf)(t) = f(-t)$. Given $\varepsilon = (\varepsilon_1,\varepsilon_2) \in \{\pm 1\}^2$, we define $\T_\R \oplus_\varepsilon \T_\R$ to  be the C*-algebra on $L^2(\R) \oplus L^2(\R)$ given as  the pullback algebra
\[ \{ (T_1,T_2) \in \T_\R \oplus \T_\R : \sigma^{\varepsilon_1}_\R(T_1) = \sigma^{\varepsilon_2}_\R(T_2) \} \]
equipped with the map $\sigma^\varepsilon_\R  :\T_\R \oplus_\varepsilon \T_\R  \to C^*(\R)$ defined by
\[ \sigma^\varepsilon_\R (T_1,T_2) \colonequals \sigma^{\varepsilon_1}_\R(T_1) = \sigma^{\varepsilon_2}_\R(T_2). \]
By definition,  $\T_\R \oplus_\varepsilon \T_\R$ sits in an exact sequence
\begin{align*}
\begin{tikzcd}[ampersand replacement=\&]
0 \ar[r] \& \K(L^2(\R)) \oplus \K(L^2(\R))  \ar[r]   \&\T_\R \oplus_\varepsilon \T_\R \ar[r,"{\sigma^\varepsilon_\R}"]  \& C^*(\R) \ar[r]  \& 0.
\end{tikzcd}
\end{align*}
One may check, using Corollary~\ref{genX}, that the K-theory boundary map 
\[ K_1(C^*(\R)) \to K_0(\K(L^2(\R)) \oplus \K(L^2(\R))) \cong \Z \oplus \Z \]
determined by this extension sends $[1-b] \mapsto \varepsilon$. 

Identity each of $\R_- \cup \{0\}$ and $\{0\} \cup \R_+$ with $\Ri$ in such a way that the  flow $\phi$ is translated to either the usual translation flow (in the case of a sink) or the time-reversed translation flow (in the case of a source). The identifications chosen lead to corresponding isometric isomorphisms $V_- : L^2(\R_-) \oplus  L^2(\R)$ and $V_+ : L^2(\R_+) \to L^2(\R)$. These choices are such that $V \colonequals V_- \oplus V_+$ carries the image of $\pi$ onto $\T_\R \oplus_\varepsilon \T_\R$. One may check that the following diagram is commutative:
\begin{align*}
\begin{tikzcd}[ampersand replacement=\&]
0 \ar[r] \&C_0(\R_- \cup \R_+) \rtimes_\phi \R  \ar[r] \ar[d,"{\mathrm{Ad}_V \circ \pi}"] \& C_0(\R) \rtimes_\phi \R \ar[r] \ar[d,"{\mathrm{Ad}_V \circ \pi}"] \& C^*(\R) \ar[r] \ar[d,equals] \& 0 \\
0 \ar[r] \& \K(L^2(\R))  \oplus \K(L^2(\R)) \ar[r] \& \T_\R \oplus_\varepsilon \T_\R    \ar[r,"{\sigma^\varepsilon_\R}"] \& C^*(\R) \ar[r] \& 0.
\end{tikzcd}
\end{align*}
The vertical maps on the left and right are isomorphisms, whence so is the map in the center, proving (1). This diagram also establishes (2), by the above calculation and naturality of the K-theory boundary map.
\end{proof}

\begin{cor}
Let $\phi$ and $\psi$ be smooth actions of $\R$ on $\R$ with $0$ as their unique fixed point.  Then the following are equivalent: 
\begin{enumerate}
\item There exists an isomorphism $C_0(\R) \rtimes_\phi \R \to C_0(\R) \rtimes_\psi \R$ and an isomorphism $\K \oplus \K \to \K \oplus \K$ preserving the order of the factors making the following diagram commutative
\[ \begin{tikzcd}
0 \ar[r] &\K \oplus \K \ar[r] \ar[d] & C_0(\R) \rtimes_\phi \R \ar[r] \ar[d]& C^*(\R) \ar[r] \ar[d,equals] & 0 \\
0 \ar[r] & \K \oplus \K \ar[r] &C_0(\R) \rtimes_\psi \R   \ar[r] &C^*(\R) \ar[r] &0
\end{tikzcd}.\]
Here, the top and bottom rows are the exact sequences coming from evaluation at $0$ as in \eqref{inducer}.
\item $\phi$ and $\psi$ have the same index, as tabulated in Figure~\ref{biindex}.
\end{enumerate}
\end{cor}
\begin{proof}
Follows from the naturality of the K-theory long exact sequence.
\end{proof}

We know from faithfulness of the representation of $C_0(\R) \rtimes_\phi \R$ on $L^2(\R_-) \oplus L^2(\R_+)$ that the ideal $I \cong \K \oplus \K$ given as the kernel of the map $C_0(\R) \rtimes_\phi \R \to C^*(\R)$ is an essential ideal. It is also, as we now show, an \emph{intrinsic} ideal in the sense that it is preserved by any $*$-automorphism of $C_0(\R) \rtimes_\phi \R$.

\begin{lemma}\label{CstarIntrinsic}
Let $A$ be a C*-algebra. Then $A$ can contain at most one one closed, $*$-invariant ideal which is an essential ideal and is  isomorphic to a finite direct sum of simple C*-algebras. In particular, if $I$ is such an ideal in $A$, then $\theta(I)=I$ for every $*$-automorphism $\theta :A \to A$.
\end{lemma}
\begin{proof}
Clearly an essential ideal must  contain any ideal which is simple as a C*-algebra. Recall that, in the C*-algebra context, ideals of ideals are again ideals in the ambient C*-algebra. Therefore an essential ideal must in fact  contain any ideal  that is isomorphic to a finite direct sum of simple C*-algebras. The conclusion follows.
\end{proof}

Let $B$ be a C*-algebra containing an ideal $I$ which is isomorphic to the direct sum of finitely-many copies of $\K$, the algebra of compact operators on a separable Hilbert space. Then we have an map 
\[ \Index  : K_1(B/I) \to \Z \]
given as the composition of the boundary map $K_1(B/I) \to K_0(I)$ associated to the short exact sequence $0 \to I \to B \to B/I \to 0$ and the canonical map $K_0(I) \to \Z$ sending the class of any minimal projection to $1$. If $I$ is also an essential ideal in $B$, then it is intrinsic by the above lemma. Therefore, any automorphism $\theta:B \to B$ descends to an automorphism of $\overline \theta :B / I \to B/I$. Naturality of the K-theory boundary map as well as invariance of the addition map $\Z^n \to \Z$ under permuting the factors, then gives that $\overline \theta_* : K_1(B/I) \to K_1(B/I)$ commutes with $\Index  : K_1(B/I) \to \Z$.

If $\phi$ is the flow of a complete vector field generating $\F^k_\R$ for $k$ odd, then $\phi$ has index $(1,1)$ or $(-1,-1)$. If $\phi$ is the flow of a complete vector field generating $\F^k_\R$ for $k$ even, then $\phi$ has index $(-1,1)$ or $(1,-1)$. Recalling that $C^*(\F^k_\R) \cong C_0(\R) \rtimes_\phi \R$, the above discussion finally gives us:

\begin{thm}
Let $k$ and $\ell$ be a positive integers. Then, $C^*(\F^k_\R)$ is isomorphic to $C^*(\F^\ell_\R)$ if and only if $k$ and $\ell$ have the same parity. 
\end{thm}

Lastly, since all the C*-algebras $C^*(\F^k_\R)$ have natural faithful representations on   $L^2(\R_-) \oplus L^2(\R_+)$, it makes sense to compare the images of these representations. Here we obtain the following:

\begin{thm}
Let $\pi_k$ denote the canonical representation of $C^*(\F^k_\R)$ on $L^2(\R_-) \oplus L^2(\R_+)$ and let $A_k$ denote the image of $\pi_k$. If $k$ and $\ell$ are distinct positive integers, then $A_k \neq A_\ell$. Indeed, $A_k \cap A_\ell = \K(L^2(\R_-)) \oplus \K(L^2(\R_+))$. 
\end{thm}
\begin{proof}
Recall that $C^*(\F^k_\R) \cong C_0(\R) \rtimes_\phi \R$  where $\phi$ is the flow of a complete vector field agreeing with $x^k \frac{d}{dx}$ near $x=0$. Similarly,  $C^*(\F^\ell_\R)  \cong C_0(\R) \rtimes_\psi \R$  where $\psi$ is the flow of a complete vector field agreeing with $x^\ell \frac{d}{dx}$ near $x=0$. 

Recall that, by choosing a diffeomorphism $u : \R_- \sqcup \R_+ \to \R \sqcup \R$ conjugating the flow $\phi$ to a pair of translation flows, we can conjugate $A_k$ to a  pushout of Weiner-Hopf algebras. We may do the same for $A_\ell$, using a diffeomorphism $v: \R_- \sqcup \R_+ \to \R \sqcup \R$. We may reformulate our problem to be about comparing these pushout algebras under  the unitary induced by the diffeomorphism $w = u \circ v^{-1}$. The key thing to notice is  that, if $k > \ell$, then this diffeomorphism must be such that $\lim_{x \to \infty} w'(x) = + \infty$ (on each copy of $\R$), due to the different asymptotics of $\phi$ and $\psi$ close to $0$. In this way, one can reduce the problem to the proposition below. 
\end{proof}

\begin{propn}
Let $u$ be an orientation-preserving diffeomorphism of $\R$. Let $U  : L^2(\R) \to L^2(\R)$ be the corresponding Hilbert space isomorphism defined by $(U \xi)(x)= (u'(x))^{1/2} \xi(u(x))$. If $\lim_{x \to \infty} u'(x) = \infty$, then the Weiner-Hopf algebra $\T_\R$ is not preserved by $U$. Indeed, $\T_\R \cap \mathrm{Ad}_U(\T_\R) = \K(L^2(\R))$.
\end{propn}
\begin{proof}
Let $P$ denote orthogonal projection onto $L^2([0,\infty)) \subset L^2(\R)$. Let $f_1,f_2 \in C_0(\R)$ so that $\widehat f_1$ and $\widehat f_2$ are arbitrary elements of $C^*(\R)$. Let $T_i = \lambda(\widehat f_i) P$, so that $T_1,T_2 \in \T_\R$ and $\sigma_\R(T_i) = \widehat f_i$. We need to show that, if $f_1$ and $f_2$ are not both zero, then $T_1 - U^{-1}T_2U$ is not compact. We consider   only the case where $f_1 \neq 0$ (the case where $f_2$ is not zero and $f_1$ is possibly zero is similar). 

Choose a continuous compactly-supported function  $\xi$ on $\R$ with $\|\xi\|_2 \neq 0$ such that $\|\lambda( \widehat f_1) \xi\|_2 = a > 0$. Since $\lambda(f_1)$ commutes with  the translation, we may assume without loss of generality that $\xi$ is supported in $[0,\infty)$. Let $b>0$ be larger than the diameter of the support of $\xi$ so that $\xi_n \colonequals \tau_{nb}(\xi)$, $n=0,1,2,\ldots$ defines an orthonormal sequence in $L^2([0,\infty))$.  For each $n$, we have $\|T_1 \xi_n\|_2 = a$. 

We  claim that $U^{-1}T_2 U \xi_n \to 0$. Once this is proved, we will have that the image by $T_1 - U^{-1}T_2 U$ of an orthonormal sequence does not converge to zero, which shows $T_1 - U^{-1}T_2U$ is not a compact operator. For $n$ sufficiently large, $U \xi_n$ is supported inside $[0,\infty)$, and so $\| U^{-1}T_2 U \xi_n\|_2 = \|\lambda(\widehat f_2) U \xi_n\|_2 = \| f_2 \eta_n\|_2$ where $\eta_n \in L^2(\R)$ is defined by $\widehat{\eta_n} = U \xi_n$. Since $\|U\xi_n\|_2 = 1$, but the diameter of the support of $U \xi_n$ goes to $0$ as $n \to \infty$, an application of the Cauchy-Schwartz inequality, gives that the $L^1$-norm of $U \xi_n$ vanishes as $n \to \infty$. On the other side of the Fourier transform, we then have that $\eta_n \in C_0(\R) \cap L^2(\R)$ and $\| \eta_n\|_\infty \to 0$.   The result will follow if we show that $\| f_2 \eta_n\| \to 0$. This follows from the following simple estimate, valid for any $R > 0$:
\[ \|f_2 \eta_n\|_2^2 \leq \|\eta_n\|_\infty^2 \left( \int_{|x|<R} |f_2(x)|^2 \ dx  \right)+  \sup_{|x| > R} |f_2(x)|^2.  \]
\end{proof}

\section{The smooth algebras of the foliations $\F^k_\R$}

We now come to the central result in  this article, that the smooth convolution algebras $\A(\F^k_\R)$, as $k$ ranges over the positive integers, are pairwise nonisomorphic. The notion of isomorphism being considered here is slightly stronger simple algebraic isomorphism.

\begin{defn}\label{IsomDefn}
Let $k$ and $\ell$ be positive integers. An \emph{isomorphism} $\A(\F^k_\R) \to \A(\F^\ell_\R)$ will refer to   a $*$-algebra isomorphism which can  furthermore be (uniquely) extended to a C*-algebra isomorphism $C^*(\A^k_\R) \to C^*(\F^\ell_\R)$. 
\end{defn}

By the results of Section~\ref{Cstarclass}, we therefore already have that $\A(\F^k_\R)$ cannot be isomorphic to $\A(\F^\ell_\R)$ if $k$ and $\ell$ do not have the same parity.

\begin{rmk}
Nearly any continuity hypothesis one might think to impose on a map at the level of smooth algebras  guarantees the existence of an extension to a map of C*-algebras. It is entirely plausible that every (algebraic) $*$-automorphism of $\A(\F^k_\R)$ is necessarily continuous, and therefore of the above form. However, we prefer not to get bogged down with delicate questions about automatic continuity.
\end{rmk}

Recall that the smooth convolution algebra $\A(\F^k_\R)$ is canonically isomorphic to $C_c^\infty(\R \rtimes_\phi \R)$, where $\phi$ denotes the flow of any complete vector field generating $\F^k_\R$. The convolution  product and adjoint operation on $C_c^\infty(\R \rtimes_\phi \R)$ are defined by as follows:
\begin{align}\label{convinv}
(f  * g)(x,t) &= \int_\R f(\phi_u(x),t-u)g(x,u) \ du \\
\nonumber f^*(x,t) &= \overline{f(\phi_t(x), -t)}.
\end{align}
Evaluation at $x=0$ leads to an exact sequence
\begin{align*}
\begin{tikzcd}[ampersand replacement=\&]
0 \ar[r] \&  C_c^\infty(\R \rtimes_\phi \R) \cdot x  \ar[r] \&  C_c^\infty(\R \rtimes_\phi \R) \ar[r] \& C_c^\infty(\R) \ar[r] \& 0 
\end{tikzcd}
\end{align*}
The algebra on the left is the ideal in $C_c^\infty(\R \rtimes_\phi \R)$ of functions which vanish to first order on $\{0\} \times \R$, the isotropy group at $0$. The logic behind the  notation for this ideal is as follows: one understands the symbol $x$ as referring to the  coordinate function of  $\R$ and then uses the fact that the algebra $C^\infty(\R)$ of smooth functions on the base manifold multiplies $C_c^\infty(\R \rtimes_\phi \R)$ from the left and the right according to the formulas below.
\begin{align*}
(f \cdot g)(x,t) &= f(\phi_t(x)) g(x,t) && \\ 
(g\cdot f)(x,t) &= f(x)g(x,t) && f \in C^\infty(\R) && g \in C_c^\infty(\R \rtimes_\phi \R).
\end{align*}
Note that this product satisfies the expected associative laws:
\begin{align*}
f\cdot (g * h) &= (f \cdot g) * h && \\ 
(g * h) \cdot f &= g *(h \cdot f) && f \in C^\infty(\R) && g,h \in C_c^\infty(\R \rtimes_\phi \R).
\end{align*}
We shall need to know the following:
\begin{propn}\label{intrinsicbasecase}
The ideal $x \cdot C_c^\infty(\R \rtimes_\phi \R)$ is ``intrinsic'' to the algebra $C_c^\infty(\R \rtimes_\phi \R)$ in the sense that any automorphism of the ambient algebra maps this ideal onto itself.
\end{propn}
\begin{proof}
Working inside the C*-completion $C_0(\R) \rtimes_\phi \R$, we have that $x \cdot C_c^\infty(\R \rtimes_\phi \R)$ is the intersection with $C_c^\infty(\R \rtimes_\phi \R)$ of the ideal $I \cong \K \oplus \K$  of Section~\ref{Cstarclass} given as the kernel of the map $C_0(\R) \rtimes_\phi \R \to C^*(\R)$. The conclusion follows from Proposition~\ref{CstarIntrinsic} and Definition~\ref{IsomDefn}.
\end{proof}

More generally, we can consider nested sequence of $*$-invariant ideals 
\[ C_c^\infty(\R \rtimes_\phi \R) \cdot x^p   \subset C_c^\infty(\R \rtimes_\phi \R) \]corresponding to the functions which vanish to $p$th order on $\{0\} \times \R$ as well as  the infinite vanishing order ideal
\[C_c^\infty(\R \rtimes_\phi \R) \cdot x^\infty \colonequals \bigcap_{p=1}^\infty  C_c^\infty(\R \rtimes_\phi \R) \cdot x^p. \] That these spaces of functions are ideals with respect to convolution follows from Corollary~\ref{DeltaId} below. One may see Section~7 of \cite{Francis[DM]} for treatment of a more general situation. Let us first recall the following definition.
 \begin{defn}
 Let $G$ be a Lie groupoid. A  smooth function $\Delta: G \to (0,\infty)$ is called if a 1-cocycle if  $\Delta(gh)=\Delta(g)\Delta(h)$ is satisfied for all composable $g,h \in G$.  
 \end{defn}

If $\Delta$ is a smooth 1-cocycle of a Lie groupoid $G$, then the pointwise product $f \mapsto \Delta f$ is an algebra automorphism for the convolution product on $C_c^\infty(G)$ determined by (any) smooth Haar system. 

\begin{propn}
Let $\phi$ be the flow of a complete vector field generating $\F^k_\R$. Then, there exists a smooth 1-cocycle $\Delta : C_c^\infty(\R \rtimes \phi \R) \to (0,\infty)$ such that $\Delta(x,t)=\frac{\phi_t(x)}{x}$ is satisfied whenever $(x,t) \in \R \rtimes_\phi \R$ and $x \neq 0$. 
\end{propn}

\begin{proof}
The vector field $X$ whose flow is $\phi$ has the form $X=h(x) x^k \frac{d}{dx}$, where $h$ is a smooth, nowhere vanishing function. We have $\phi'_t(x) = \frac{h(\phi_t(x)) (\phi_t(x))^k}{h(x)x^k}$ whenever $x \neq 0$, and so
\[ \frac{\phi_t(x)}{x} =  \left( \frac{h(x)}{h(\phi_t(x))} \phi_t'(x) \right)^{1/k} \]
for $x \neq 0$. The expression on the right defines a smooth, positive-valued function $(x,t) \mapsto \Delta(x,t)$ on all of $\R \rtimes_\phi \R$. The expression on the left gives us that $\Delta$ is multiplicative on the dense subgroupoid of $\R \rtimes_\phi \R$ consisting of $(x,t)$ with $x \neq 0$. By continuity, $\Delta$ is multiplicative everywhere.
\end{proof}

\begin{cor}\label{DeltaId}
In the context of the proposition above, for every positive integer $p$, we have that $C_c^\infty(\R \rtimes_\phi \R) \cdot x^p = x^p  \cdot C_c^\infty(\R \rtimes_\phi \R)$ is a $*$-invariant ideal in $C_c^\infty(\R \rtimes_\phi \R)$. 
\end{cor}
\begin{proof}
The above proposition gives the relation $x \cdot f = ( \Delta f) \cdot x$ for every $f \in C_c^\infty(\R \rtimes_\phi \R)$. From this, the result follows. 
\end{proof}

In fact the sequence of ideals $x^p \cdot C_c^\infty(\R \rtimes_\phi \R)$ is also intrinsic. To see this, we need the following Dixmier-Malliavin theorem.

\begin{propn}\label{DMspecial}
The smooth convolution algebra $C_c^\infty(\R \rtimes_\phi \R)$ satisfies $C_c^\infty(\R \rtimes_\phi \R)*C_c^\infty(\R \rtimes_\phi \R)=C_c^\infty(\R \rtimes_\phi \R)$. 
\end{propn}
\begin{proof}
This is a special case of Theorem~2 in \cite{Francis[DM]}. In fact, rather than using the full force of Theorem~2, one can instead use \cite{Francis[DM]}, Theorem~3 (which avoids the most technical parts of that paper) to write  any $f \in C_c^\infty(\R \rtimes_\phi \R)$ as $f = g_0 * h_0 + g_1 * h_1$ where $g_0,g_1 \in C_c^\infty(\R)$ and $h_0,h_1 \in C_c^\infty(\R \rtimes_\phi \R)$. Left multiplying by a suitable cutoff function then gives a representation of $f$ as a two-term sum in which each term is the convolution of functions in $C_c^\infty(\R \rtimes_\phi \R)$. 
\end{proof}

\begin{cor}\label{smoothintrinsic}
For any automorphism $\theta$  of $C_c^\infty(\R \rtimes_\phi \R)$ and any $p=1,2,\ldots, \infty$, we have  $\theta(x^p \cdot C_c^\infty(\R \rtimes_\phi \R))=x^p \cdot C_c^\infty(\R \rtimes_\phi \R)$.
\end{cor}
\begin{proof}
The case $p=1$ is Proposition~\ref{intrinsicbasecase}. For $p$ a positive integer, it follows from writing $(x \cdot C_c^\infty(\R \rtimes_\phi \R))^{*p} = x^p * (C_c^\infty(\R \rtimes_\phi \R))^{*p} = x^p \cdot C_c^\infty(\R \rtimes_\phi \R)$ where the last equality uses Proposition~\ref{DMspecial}.

\end{proof}

The following result is the means by which we distinguish the different $\A(\F^k_\R)$.

\begin{thm}\label{commVSnoncomm}
Let $k$ and $p$ be positive integers. Then, the quotient of $\A(\F^k_\R)$ by the ideal $x^p \A(\F^k_\R)$ is commutative for $p \leq k$ and noncommutative for $p>k$.
\end{thm}

Taken together, Theorem~\ref{commVSnoncomm} and Corollary~\ref{smoothintrinsic} imply

\begin{thm}
The smooth convolution algebras $\A(\F^k_\R)$, $k$ a positive integer, are pairwise nonisomorphic.
\end{thm}

It remains to establish  Theorem~\ref{commVSnoncomm}. We prove the following more precise result which has Theorem~\ref{commVSnoncomm} as a corollary.

\begin{thm}
Fix positive integers $p$ and $k$. Let $X$ be a complete vector field  generating $\F^k_\R$ and assume for simplicity that $X$ coincides with $x^k \frac{d}{dx}$ on a neighbourhood of $0$. Let $\phi$ be the flow of $X$ and form the transformation groupoid $\R \rtimes_\phi \R$ with usual coordinates $(x,t)$. Let 
\[ T :  C_c^\infty(\R \rtimes_\phi \R) \to \frac{C_c^\infty(\R) [x]}{(x^{p+1})} \]
be the surjective linear map which assigns to a function $f \in C_c^\infty(\R \rtimes_\phi \R)$ its $p$th order Taylor expansion $T(f) = f_0 + f_1 x+ \ldots f_p x^p$, where $f_n(t) = \frac{1}{n!} \frac{\partial^n f}{\partial x^n}(0,t)$. Since the kernel of $T$ is $x^{p+1} \cdot \A(\F^k_\R)$, we may use the map $T$ to transfer the ring structure of $\frac{C_c^\infty(\R \rtimes_\phi \R)}{x^{p+1} \cdot C_c^\infty(\R \rtimes_\phi \R)}$  to  $C_c^\infty(\R)[x]/(x^{p+1})$. 
\begin{enumerate}
\item If $p < k$, then the product  on $C_c^\infty(\R)[x]/(x^{p+1})$ induced by $T$  is the usual one given by multiplication followed by truncation (the product on the coefficient ring $C_c^\infty(\R)$ is convolution). 
\item  If $p=k \geq 2$, then the product on $C_c^\infty(\R)[x]/(x^{k+1})$  induced by $T$ is determined by the single relation $xf = fx + \delta(f)x^k$, where $\delta : C_c^\infty(\R) \to C_c^\infty(\R)$ is the derivation defined by $\delta(f)(t) = tf(t)$.
\item If $p=k=1$, then the product on $C_c^\infty(\R)[x]/(x^2)$  induced by $T$ is determined by the single relation $xf = \Delta(f)x$, where $\Delta : C_c^\infty(\R) \to C_c^\infty(\R)$ is the algebra automorphism defined by $\Delta(f)(t) = e^tf(t)$.
\end{enumerate}
\end{thm}
\begin{proof}
We shall only consider the (most complicated) case $p=k \geq 2$. It is convenient to enlarge $C_c^\infty(\R  \rtimes_\phi \R)$ slightly and work instead in the \emph{smooth translation algebra} $C^\infty_u(\R \rtimes_\phi \R)$ consisting of smooth functions $f$ on $\R \rtimes_\phi \R$ supported inside $M \times [-r,r]$ for some $r >0$ (depending on $f$). The convolution product and adjoint operation on $C^\infty_u(\R \rtimes_\phi M)$ are defined in the same way, using \eqref{convinv}. The translation algebra has the convenient property of containing $C_c^\infty(\R)$ as a subalgebra. Specifically, given $f \in C_c^\infty(\R)$, define $\widetilde f \in C_u^\infty(\R\rtimes_\phi \R)$ by $\widetilde f(x,t)=f(t)$ for all $(x,t)$. The assignment 
\[ f \mapsto \widetilde f : C_c^\infty(\R) \to C_u^\infty(\R \rtimes_\phi \R) \] 
is an injective $*$-homomorphism satisfying $T(\widetilde f) =f$ for $f \in C_c^\infty(\R)$. We also have the obvious relation $T(f \cdot x) = T(f) x$, $f \in C_u^\infty(\R \rtimes_\phi \R)$, whence we may extend  $f \mapsto \widetilde f$ to a linear splitting $S: C_c^\infty(\R)[x]/(x^{k+1}) \to C_u^\infty(\R \rtimes_\phi \R)$ of $T$ defined as follows:
\begin{align*}
 S(f_0 + f_1 x + \ldots + f_k x^k) =  \widetilde f_0   + \widetilde f_1 \cdot x + \ldots + \widetilde f_k \cdot x^k  
\end{align*}
It therefore remains to understand the product on elements of the form $\widetilde f_0   + \widetilde f_1 \cdot x + \ldots + \widetilde f_k \cdot x^k$, modulo the ideal $x^{k+1} \cdot C_u^\infty(\R \rtimes_\phi \R)$.   Let $\widetilde \delta$ be the derivation of $C_u^\infty(\R \rtimes_\phi \R)$ given by $(\widetilde \delta f)(x,t) = t f(x,t)$. Then we have 
 \begin{align*} 
 \widetilde{\delta f}=\widetilde \delta \widetilde f && f \in C_c^\infty(\R).
 \end{align*}
Recall that $\F^k_\R$ is singly-generated by $x^k \frac{d}{dx}$. When $k =1$, this is just the Euler vector field; its flow is complete and given  by the scaling action $\phi_t(x) =  e^t x$. When   $k \geq 2$, the flow is not complete. The flow is given by
\begin{align*}
\phi_t(x)=\frac{x}{\sqrt[k-1]{1-(k-1)tx^{k-1}}}.
\end{align*}
and is  defined on the domain $\{ (t,x)  : tx^{k-1} < \tfrac{1}{k-1} \}$. More important for us than this specific formula is the way that the Taylor series of the flow starts:
\begin{align}
\phi_t(x) \sim x+ t x^k + \ldots 
\end{align}
  Since $X$ is a complete vector field generating $\F^k_\R$, $k \geq 2$ which furthermore coincides with $x^k \frac{d}{dx}$ on a neighbourhood of $0$, it follows that the flow of $\phi$ can be expressed as
\[ \phi_t(x) = x + t x^k + h(x,t) \]
where $h$ is a smooth function vanishing to order $k+1$ on $\{0\} \times \R$.  We therefore have, for any $f \in C_u^\infty(\R \rtimes_\phi \R)$,
\[ x \cdot f = f \cdot x + (\widetilde \delta f) \cdot x^k + r \]
where $r \in x^{k+1} \cdot C_c^\infty(\R \rtimes_\phi \R)$.  
Replacing $f$ above with $\widetilde f$ for $f \in C_c^\infty(\R)$, we obtain 
\begin{align*}
x \cdot \widetilde f 
=\widetilde f \cdot x + \widetilde{\delta(f)} \cdot x^k  \pmod{x^{k+1} \cdot C_u^\infty(\R \rtimes_\phi \R)}.
\end{align*}
Together with the fact that $\widetilde f * \widetilde g  = \widetilde{f*g}$, for $f,g \in C_c^\infty(\R)$, the above determines the product on $C_c^\infty(\R)[x]/(x^{k+1})$ to be as claimed.
\end{proof}

Finally, we describe the product structure on $\frac{C_c^\infty(\R \rtimes_\phi \R)}{x^\infty \cdot C_c^\infty(\R \rtimes_\phi \R)}$. According to Borel's theorem (\cite{Hormander}, pp~16),  the mapping $T:  C_c^\infty(\R \rtimes_\phi \R) \to C_c^\infty(\R)[[x]]$ to formal series with coefficients in $C_c^\infty(\R)$ given by forming the Taylor expansion at $x=0$ is surjective. Clearly its kernel is equal to $x^\infty \cdot C_c^\infty(\R \rtimes_\phi \R)$, so we may understand the product on $\frac{C_c^\infty(\R \rtimes_\phi \R)}{x^\infty \cdot C_c^\infty(\R \rtimes_\phi \R)}$ by transferring it to  $C_c^\infty(\R)[[x]]$.

Some special notations will be needed. Given a smooth action $\phi$ of $\R$ on $\R$ fixing the origin, let
\[ \phi_t(x) \sim \sum_{m=1}^\infty \phi_m(t) x^m \]
be the Taylor expansion of the flow in $x$, so that $\phi_m$ is a smooth function on $\R$. More generally, we consider the Taylor expansion of any power of the flow: 
\[ (\phi_t(x))^n \sim  \sum_{m=n}^\infty \phi_m^n(t) x^m. \]
The smooth functions $\phi_m^n$ will be used to describe the product on $C_c^\infty(\R)[[x]]$ induced by $T$.

\begin{propn}
Let $\phi$ be a smooth action of $\R$ on $\R$ which fixes the origin.  The product $*$ on $C_c^\infty(\R)[[x]]$ determined by the Taylor expansion map $T:  C_c^\infty(\R \rtimes_\phi \R) \to C_c^\infty(\R)[[x]]$   is such that, if  $f = \sum_{n \geq 0} f_n x^n$ and $g = \sum_{n \geq 0} g_n x^n$, then $f*g = h$ where $h = \sum_{n \geq 0} h_n x^n$ is given by
\[ h_p = \sum_{n \leq m \leq p} f_n * (\phi_m^n g_{p-m}).\]
\end{propn}
Notice that if $\phi_t(x)=x$, or even if $\phi$ simply fixes the origin to infinite order,  then $\phi^m_n(t)=1$ if $m=n$ and $0$ otherwise, restoring the usual (commutative) power structure on $C_c^\infty(\R)[[x]]$. 
\begin{proof}
Throughout, $[m]$ denotes the operation of extracting the coefficient of $x^m$ in a formal series. Let $F,G \in C_c^\infty(\R \rtimes_\phi \R)$ be lifts of $f,g$ through the map $T: C_c^\infty(\R \rtimes_\phi \R) \to \C[[x]]$. Let $H = F*G$. Then, we have
\begin{align*}
[p] F(t-s,\phi_s(x)) G(s,x) &= \sum_{m=0}^p \left( [m] F(t-s,\phi_s(x)) \right) \left( [p-m] G(s,x) \right) \\
&=\sum_{m=0}^p \sum_{n=0}^m \left( [m] \phi_s(x)^n \right) \left( [n]  F(t-s,x) \right) g_{p-m}(s) \\
&=\sum_{m=0}^p \sum_{n=0}^m \phi_m^n(s) f_n(t-s) g_{p-m}(s) 
\end{align*}
Integrating with respect to $s$ then gives
\[[p] H(x,t) = \sum_{m=0}^p \sum_{n=0}^m (f_n * (\phi_m^n g_{p-m})(t) \]
\end{proof}

We conclude by pointing out the following relation among the functions $\phi_m^n$ which are being used to define the product on $C_c^\infty(\R)[[x]]$.

\begin{propn}
Let $\phi$ be  a smooth action of $\R$ on $\R$ fixing the origin. As above, define smooth functions $\phi_m^n$, for $m, n \geq 0$, by $\phi_m^n(t) = [m] (\phi_t(x))^n$, where $[m]$ denotes the operation of extracting the coefficient of $x^m$ in a Formal series. Then, the following identities are satisfied.
\[ \phi_m^n(t) = \sum_{i=n}^m \phi_m^i(t-s) \phi_i^n(s) \]\end{propn}
\begin{proof}
For any smooth functions $f$ and $g$ on $\R$ which vanish at $0$, one has
\begin{align*}
[m] (f \circ g)^n = \sum_{i=n}^m ([m] g^i)([i] f^n).
\end{align*}
Taking $f = \phi_{t-s}$ and $g=\phi_s$ so that $f \circ g = \phi_t$ we get the desired identity. 
\end{proof}

This identity above suggests the presence of a Hopf  algebra behind the scenes. It would be interesting to look for a  connection between the algebras $\frac{C_c^\infty(\R \rtimes_\phi \R)}{x^\infty \cdot C_c^\infty(\R \rtimes_\phi \R)}$ and the the work of Connes and Kreimer \cite{Connes-Kreimer}. See also \cite{Butcher}, \cite{Cayley}.

\bibliographystyle{abbrv}
\bibliography{ApplBib}

\begin{thebibliography}{10}

\bibitem{AS[2007]}
I.~Androulidakis and G.~Skandalis.
\newblock The holonomy groupoid of a singular foliation.
\newblock {\em J. Reine Angew. Math.}, 626:1--37, 2009.

\bibitem{Butcher}
J.~C. Butcher.
\newblock Coefficients for the study of {R}unge-{K}utta integration processes.
\newblock {\em J. Austral. Math. Soc.}, 3:185--201, 1963.

\bibitem{Cayley}
A.~Cayley.
\newblock On the {A}nalytical {F}orms {C}alled {T}rees.
\newblock {\em Amer. J. Math.}, 4(1-4):266--268, 1881.

\bibitem{Connes[Thom]}
A.~Connes.
\newblock An analogue of the {T}hom isomorphism for crossed products of a
  {$C^{\ast} $}-algebra by an action of {${\bf R}$}.
\newblock {\em Adv. in Math.}, 39(1):31--55, 1981.

\bibitem{Connes-Kreimer}
A.~Connes and D.~Kreimer.
\newblock Hopf algebras, renormalization and noncommutative geometry.
\newblock {\em Comm. Math. Phys.}, 199(1):203--242, 1998.

\bibitem{Debord}
C.~Debord.
\newblock Holonomy groupoids of singular foliations.
\newblock {\em J. Differential Geom.}, 58(3):467--500, 2001.

\bibitem{Devinatz}
A.~Devinatz.
\newblock On {W}iener-{H}opf operators.
\newblock In {\em Functional {A}nalysis ({P}roc. {C}onf., {I}rvine, {C}alif.,
  1966)}, pages 81--118. Academic Press, London; Thompson Book Co., Washington,
  D.C., 1967.

\bibitem{Dixmier-Malliavin}
J.~Dixmier and P.~Malliavin.
\newblock Factorisations de fonctions et de vecteurs ind\'{e}finiment
  diff\'{e}rentiables.
\newblock {\em Bull. Sci. Math. (2)}, 102(4):307--330, 1978.

\bibitem{Douglas}
R.~G. Douglas.
\newblock {\em Banach algebra techniques in the theory of {T}oeplitz
  operators}.
\newblock American Mathematical Society, Providence, R.I., 1973.
\newblock Expository Lectures from the CBMS Regional Conference held at the
  University of Georgia, Athens, Ga., June 12--16, 1972, Conference Board of
  the Mathematical Sciences Regional Conference Series in Mathematics, No. 15.

\bibitem{Francis[DM]}
M.~{Francis}.
\newblock {A Dixmier-Malliavin theorem for Lie groupoids. arXiv:2009.13760}.
\newblock Sept. 2020.

\bibitem{Green}
P.~Green.
\newblock {$C\sp*$}-algebras of transformation groups with smooth orbit space.
\newblock {\em Pacific J. Math.}, 72(1):71--97, 1977.

\bibitem{Hormander}
L.~H\"{o}rmander.
\newblock {\em The analysis of linear partial differential operators. {I}},
  volume 256 of {\em Grundlehren der Mathematischen Wissenschaften [Fundamental
  Principles of Mathematical Sciences]}.
\newblock Springer-Verlag, Berlin, second edition, 1990.
\newblock Distribution theory and Fourier analysis.

\bibitem{Renault[BOOK]}
J.~Renault.
\newblock {\em A groupoid approach to {$C^{\ast} $}-algebras}, volume 793 of
  {\em Lecture Notes in Mathematics}.
\newblock Springer, Berlin, 1980.

\bibitem{Renault}
J.~Renault.
\newblock Cartan subalgebras in {$C^*$}-algebras.
\newblock {\em Irish Math. Soc. Bull.}, (61):29--63, 2008.

\bibitem{Rieffel}
M.~A. Rieffel.
\newblock Connes' analogue for crossed products of the {T}hom isomorphism.
\newblock In {\em Operator algebras and {$K$}-theory ({S}an {F}rancisco,
  {C}alif., 1981)}, volume~10 of {\em Contemp. Math.}, pages 143--154. Amer.
  Math. Soc., Providence, R.I., 1982.

\bibitem{Torpe}
A.~M. Torpe.
\newblock {$K$}-theory for the leaf space of foliations by {R}eeb components.
\newblock {\em J. Funct. Anal.}, 61(1):15--71, 1985.

\bibitem{Williams}
D.~P. Williams.
\newblock {\em Crossed products of {$C{^\ast}$}-algebras}, volume 134 of {\em
  Mathematical Surveys and Monographs}.
\newblock American Mathematical Society, Providence, RI, 2007.

\end{thebibliography}

\end{document}